\documentclass{amsart}
\usepackage{amssymb,mathtools}
\usepackage[british]{babel}
\usepackage{enumitem}
\usepackage[latin1]{inputenc}

\usepackage{hyperref}

\newtheorem{theorem}{Theorem}
\numberwithin{theorem}{section}
\newtheorem{corollary}[theorem]{Corollary}
\newtheorem{lemma}[theorem]{Lemma}
\newtheorem{proposition}[theorem]{Proposition}
\theoremstyle{definition}
\newtheorem{definition}[theorem]{Definition}
\newtheorem{remark}[theorem]{Remark}

\newtheorem*{claim}{Claim}

\numberwithin{equation}{section}

\title{More conservativity for weak K\H{o}nig's lemma}
\author{Anton Freund}
\author{Patrick Uftring}

\address{University of W\"urzburg, Institute of Mathematics, Emil-Fischer-Str.~40, 97074 W\"urz\-burg, Germany}
\email{\{anton.freund,patrick.uftring\}@uni-wuerzburg.de}

\thanks{Funded by the Deutsche Forschungsgemeinschaft (DFG, German Research Foundation) -- Project number 460597863.}

\begin{document}

\begin{abstract}
We prove conservativity results for weak K\H{o}nig's lemma that extend the celebrated result of Harrington (for $\Pi^1_1$-statements) and are somewhat orthogonal to the extension by Simpson, Tanaka and Yamazaki (for statements of the form~$\forall X\exists!Y\psi$ with arithmetical~$\psi$). In particular, we show that $\mathsf{WKL}_0$ is conservative over~$\mathsf{RCA}_0$ for well-ordering principles. We also show that compactness (which characterizes weak K\H{o}nig's lemma) is dispensable for certain results about continuous functions with isolated singularities.
\end{abstract}


\maketitle

\section{Introduction}

Weak K\H{o}nig's lemma is, essentially, the statement that $[0,1]\subseteq\mathbb R$ is compact (see~\cite{simpson09} for all claims in the present paragraph). The compactness of~$[0,1]$ is ineffective in the sense that it leads to sets $X\subseteq\mathbb N$ for which $n\in X$ cannot be decided by a computer program. However, it turns out that all compactness proofs of sufficiently concrete theorems can be effectivized. This can be made precise in the framework of reverse mathematics. The latter allows us to compare the logical strength of axioms and theorems about finite and countable objects (which are coded by elements and subsets of~$\mathbb N$). We note that the countable objects include, e.\,g., continuous functions on~$\mathbb R$, which are determined by their countably many values on the rationals. Most often, the comparisons of reverse mathematics take place over a basic axiom system~$\mathsf{RCA}_0$, which may be identified with effective mathematics (\emph{cum grano salis}). The system~$\mathsf{WKL}_0$ results from~$\mathsf{RCA}_0$ when we add weak K\H{o}nig's lemma as an axiom. For the purpose at hand, a statement is sufficiently concrete if it is~$\Pi^1_1$, i.\,e., of the form~$\forall X\subseteq\mathbb N:\psi(X)$ for arithmetical~$\psi$ (which means that the quantifiers in~$\psi$ may range over elements but not over subsets of~$\mathbb N$). Now the claim that compactness proofs can be effectivized is made precise by a celebrated result of L.~Harrington, which says that $\mathsf{WKL}_0$ is $\Pi^1_1$-conservative over (i.\,e., proves the same $\Pi^1_1$-statements as) the system~$\mathsf{RCA}_0$. One even has conservativity for statements of the form $\forall X\exists!Y\,\psi(X,Y)$ with arithmetical~$\psi$ (partially even for~$\psi\in\Pi^1_1$), where $\exists!$ denotes unique existence, as shown by S.~Simpson, K.~Tanaka and T.~Yamazaki~\cite{simpson-tanaka-yamazaki} (previously for~$\psi\in\Sigma^0_3$ by A.~Fernandes~\cite{fernandes-wkl}). Similar phenomena are important in proof mining~\cite{kohlenbach-book}, an approach that uses methods from logic to extract concrete information from \emph{a~priori} ineffective proofs.

The present paper proves conservativity results that extend the one by Harrington and are somewhat orthogonal to the one by Simpson, Tanaka and Yamazaki. In Section~\ref{sect:wo-princ}, we show that $\mathsf{WKL}_0$ is conservative over~$\mathsf{RCA}_0$ for statements that have the form
\begin{equation}\label{eq:lin-Pi11}
    \forall X\big(\text{``$X$ is a well-order"}\to\psi(X)\big)\quad\text{with}\quad\psi\in\Pi^1_1.
\end{equation}
We also prove a conservativity result for instances of the ascending descending sequence principle (see Corollary~\ref{cor:ads}). After the present paper was completed, the authors became aware that a `pointwise' version of conservativity for statements of the form~(\ref{eq:lin-Pi11}) was given by A.~Kreuzer and K.~Yokoyama~\cite{kreuzer-yokoyama} (who label this pointwise version as folklore). Namely, Theorem~3.2 of~\cite{kreuzer-yokoyama} says that the axiom system $\mathsf{WKL}_0+\text{``$X$ is a well-order"}$ is $\Pi^1_1$-conservative over $\mathsf{RCA}_0+\text{``$X$ is a well-order"}$ for each primitive recursive~$X$.

Since well-foundedness is a $\Pi^1_1$-property, statement (\ref{eq:lin-Pi11}) has complexity $\Pi^1_2$ (i.\,e., can be written as $\forall X\exists Y\psi$ with arithmetical~$\psi$). An important special case is provided by so-called well-ordering principles. These are statements of the form
\begin{equation}\label{eq:wo-princ}
    \forall X\big(\text{``$X$ is a well-order"}\to\text{``$D(X)$ is a well-order"}\big),
\end{equation}
where~$D$ is a computable transformation of linear orders. As an example, we mention the case where~$D(X)$ consists of the finite decreasing sequences in~$X$, ordered lexicographically (think of Cantor normal forms with exponents from~$X$). In this case, (\ref{eq:wo-princ}) is equivalent to arithmetical comprehension over~$\mathsf{RCA}_0$ (and hence unprovable in~$\mathsf{WKL}_0$), as shown by J.-Y.~Girard~\cite{girard87} and J.~Hirst~\cite{hirst94}. Many important principles above arithmetical comprehension have also been characterized by well-ordering principles (see~\cite{marcone-montalban,rathjen-afshari,rathjen-weiermann-atr,rathjen-atr,rathjen-model-bi,thomson-rathjen-Pi-1-1,freund-computable,freund-rathjen-iterated-Pi11}). Indeed, any $\Pi^1_2$-statement is equivalent to one of the form~(\ref{eq:wo-princ}) in the presence of arithmetical comprehension. Here one can even demand that~$D$ belongs to a class of particularly uniform well-ordering principles that are known as dilators (see Appendix~8.E of~\cite{girard-book-part2}). Well-ordering principles have found applications, e.\,g., in the reverse mathematics of Fra\"iss\'e's conjecture~\cite{marcone-montalban-fraisse} and related better-quasi-orders~\cite{freund-3-bqo}. Below arithmetical comprehension (and in particular in the so-called reverse mathematics zoo), $\Pi^0_2$-induction seems to be the only statement for which a characterization by a well-ordering principle is known~\cite{uftring-etr}.

Our conservativity result entails, in particular, that weak K\H{o}nig's lemma cannot be characterized by a well-ordering principle. This was first shown by the second author (see Corollary~4.1.20 of~\cite{uftring-thesis}). The first author then used completely different methods to show that many principles from the zoo cannot be characterized by dilators~\cite{freund-zoo}. Concerning~(\ref{eq:lin-Pi11}), we note that the linearity of~$X$ is crucial. As we show in Section~\ref{sect:wo-princ}, the conservativity result becomes false when one replaces ``$X$~is a well-order" by ``$X$ is a well-founded relation" or ``$X$ is a well-partial-order". This is interesting insofar as well-partial-orders seem close to being linear (their antichains are finite). For different reasons, though, weak K\H{o}nig's lemma can still not be characterized by a dilator on well-partial-orders (see Proposition~5.3.6 of~\cite{uftring-thesis}).

In Section~\ref{sect:isolated}, we prove results about the isolated existence quantifier that is explained by
\begin{equation*}
    \exists^iY\,\varphi(Y)\quad\Leftrightarrow\quad\exists Y\big(\varphi(Y)\land\exists n\forall Z(Y[n]=Z[n]\land\varphi(Z)\to Y=Z)\big),
\end{equation*}
where we write $Y[n]=\langle\chi_Y(0),\ldots,\chi_Y(n-1)\rangle$ for $Y\subseteq\mathbb N$ with characteristic function~$\chi_Y$. We show that $\mathsf{WKL}_0$ is conservative over~$\mathsf{RCA}_0$ for statements of the form $\forall X\exists^iY\,\varphi$ with arithmetical $\varphi$. As we will see, this readily follows from a technical theorem of Simpson, Tanaka and Yamazaki~\cite{simpson-tanaka-yamazaki}, though it is more general than the conservativity results that these authors state. We will obtain a new and considerably simpler proof for~$\varphi\in\Sigma^0_3$, which exploits the same idea as our result on well-ordering principles.

As an application, we consider the statement that all continuous~$f:[0,1]\to\mathbb R$ are bounded, which is equivalent to weak K\H{o}nig's lemma (i.\,e., to compactness) and hence unprovable in~$\mathsf{RCA}_0$ (see~\cite{simpson09}). As explained by Simpson, Tanaka and Yamazaki~\cite{simpson-tanaka-yamazaki}, this does not contradict conservativity for statements~$\forall X\exists!Y\psi$ (even when we make the bound unique by taking the supremum), since the assumption that $f$ is defined on all of~$[0,1]$ introduces an existential quantifier without a unique witness. We show that, nevertheless, our conservativity result for ``isolated existence" yields meaningful information: Consider a $\Sigma^1_1$-class~$\mathcal F$ of continuous functions $f:D_f\to\mathbb R$ such that $[0,1]\backslash D_f$ consists of isolated points, provably in~$\mathsf{RCA}_0$ (e.\,g., think of meromorphic functions). Then $\mathsf{RCA}_0$ proves that all~$f\in\mathcal F$ with $D_f=[0,1]$ are bounded. So even though the boundedness principle is intimately linked to compactness, the latter is in some sense dispensable for large classes of functions.

On a technical level, our results rely on the fact that countable models of~$\mathsf{RCA}_0$ have hyperimmune-free $\omega$-extensions that satisfy~$\mathsf{WKL}_0$. In fact, the forcing extensions that are constructed in typical proofs of Harrington's result are automatically hyperimmune-free. In order to make our paper more accessible, we present a proof of this fact in Section~\ref{sect:hyp-free}. Experts may skip this section if they are aware of the result, which is a straightforward variant of the hyperimmune-free basis theorem~\cite{jockusch-soare} over non-$\omega$-models.

\subsection*{Acknowledgements} We are grateful to Denis Hirschfeldt for advice and encouragement at an early stage of our project, to Ulrich Kohlenbach for important information about the literature, and to Leszek Ko{\l}odziejczyk and Keita Yokoyama for stimulating questions and comments.

\section{Hyperimmune-free $\omega$-extensions}\label{sect:hyp-free}

In this section, we show a version of the hyperimmune-free basis theorem for non-$\omega$-models. Our presentation of forcing follows Section~7 of~\cite{dzhafarov-mummert}. Experts can skip much of the section, which we aim to keep concise but reasonably self-contained.

When $\mathcal M$ is a model of second-order arithmetic, let~$M$ denote the first-order part and write $X\in\mathcal M$ to convey that~$X$ lies in the second-order part. By a function of~$\mathcal M$, we mean a function~$f:M\to M$ with $\{\langle x,y\rangle\,|\,f(x)=y\}\in\mathcal M$, where the Cantor codes~$\langle x,y\rangle$ are computed in~$\mathcal M$ (assuming that the latter satisfies basic arithmetic). For functions $f,g$ of~$\mathcal M$, we say that $g$ dominates~$f$ if we have $f(x)\leq^{\mathcal M}g(x)$ for all~$x\in M$. Let us recall that an $\omega$-extension of~$\mathcal M$ is a model~$\mathcal N$ with the same first-order part such that $X\in\mathcal M$ entails $X\in\mathcal N$. 

\begin{definition}
An $\omega$-extension of~$\mathcal M$ into~$\mathcal N$ is hyperimmune-free if each function of~$\mathcal N$ is dominated by some function of~$\mathcal M$.
\end{definition}

Note that domination is evaluated in the shared first-order part. We now recall some notation that is related to weak K\H{o}nig's lemma. Let $2^{<\omega}$ denote the set of finite sequences~$\sigma=\langle\sigma_0,\ldots,\sigma_{|\sigma|-1}\rangle$ with entries~$\sigma_i\in\{0,1\}$ (where the empty sequence arises for length $|\sigma|=0$). We write $\sigma\sqsubset\tau$ to express that $\sigma$ is a proper initial segment of~$\tau$, i.\,e., that we have $|\sigma|<|\tau|$ and $\sigma_i=\tau_i$ for all~$i<|\sigma|$. A subset~$T\subseteq 2^{<\omega}$ is called a tree if~$\sigma\sqsubset\tau\in T$ entails~$\sigma\in T$. Weak K\H{o}nig's lemma is the statement that any infinite tree~$T\subseteq 2^{<\omega}$ admits a path, i.\,e., a function~$f:\mathbb N\to\{0,1\}$ with the property that $f[n]=\langle f(0),\ldots,f(n-1)\rangle\in T$ holds for all~$n\in\mathbb N$.

For most of the present section, we fix a countable model~$\mathcal M\vDash\mathsf{RCA}_0$. Our aim is to find a hyperimmune-free $\omega$-extension into a model of~$\mathsf{WKL}_0$. This is achieved by a forcing construction that adds a path to one tree at a time. So let us also fix a~$T\in\mathcal M$ with
\begin{equation*}
\mathcal M\vDash\text{``$T\subseteq 2^{<\omega}$ is an infinite tree"}.
\end{equation*}
We employ Jockusch-Soare forcing as described in~\cite{dzhafarov-mummert}, though we tighten the presentation for the case at hand. Let $\mathbb P$ be the collection of all pairs $(n,U)$ with $n\in M$ and~$U\in\mathcal M$ such that we have
\begin{equation*}
\mathcal M\vDash\text{``$U\subseteq T$ is an infinite tree with a single sequence of length~$n$"}.
\end{equation*}
To get a partial order on~$\mathbb P$, we declare that $(m,U)\preceq(n,V)$ holds precisely when we have $m\geq^\mathcal M n$ and $U\subseteq V$. For $(n,U)\in\mathbb P$, we declare that $\mathbb V(n,U)\in U$ is the unique sequence with $\mathcal M\vDash|\mathbb V(n,U)|=n$. It is straightforward to see that we have
\begin{equation}\label{eq:force-prop1}
p\preceq q\text{ in }\mathbb P\quad\Rightarrow\quad\mathbb V(q)\sqsubseteq^\mathcal M\mathbb V(p),
\end{equation}
which is one of two properties that a forcing notion over a model should have according to Definition~7.6.1 of~\cite{dzhafarov-mummert}. The other property requires that
\begin{equation}\label{eq:force-prop2}
\text{all $q\in\mathbb P$ and $m\in M$ admit a $p\preceq q$ with $\mathcal M\vDash|\mathbb V(p)|\geq m$}.
\end{equation}
To show that this holds, we write $q=(n,V)$, where we may assume~$m>^\mathcal M n$. It~is enough to find a~$\sigma\in V$ with $\mathcal M\vDash|\sigma|=m$ such that the set
\begin{equation*}
U_\sigma=\{\tau\in V\,|\,\tau\sqsubseteq^{\mathcal M}\sigma\text{ or }\sigma\sqsubseteq^{\mathcal M}\tau\}
\end{equation*}
is $\mathcal M$-infinite, as we can then take~$p=(m,U_\sigma)\in\mathbb P$. If there was no such~$\sigma$, we would have
\begin{equation*}
\mathcal M\vDash\forall\sigma\in 2^m\exists k\forall\tau\in 2^{m+k}(\sigma\sqsubseteq\tau\to\tau\notin V),
\end{equation*}
where $2^l\subseteq 2^{<\omega}$ consists of the sequences of length~$l$. By bounded collection in~$\mathcal M$, we would then obtain a single~$k$ that works for all~$\sigma\in 2^m$, so that already~$V$ would be $\mathcal M$-finite, against the the assumption that~$q=(n,V)$ lies in~$\mathbb P$. To make the crucial point more explicit, we only need bounded collection since being finite is $\Sigma^0_1$ for subtrees of~$2^{<\omega}$ (``some level of the tree is empty") while it is~$\Sigma^0_2$ in general.

We write $\mathcal L_1^{\mathcal M}$ for the language of first-order arithmetic with (constants and predicate symbols for) number and set parameters from~$\mathcal M$. We also write $\mathcal M$ for the obvious~$\mathcal L_1^{\mathcal M}$-structure. Let $\mathcal L_1^{\mathcal M}(\mathsf G)$ be the extension of $\mathcal L_1^{\mathcal M}$ by a fresh predicate symbol~$\mathsf G$. We assume that $\mathcal L_1^{\mathcal M}(\mathsf G)$-formulas are built from atoms by negation and disjunction as well as bounded and unbounded existential quantification (so that $\exists x<t\,\psi$ is a proper formula and not an abbreviation for $\exists x\,\neg(x\geq t\lor\neg\psi)$). The following definition coincides with one from~\cite{dzhafarov-mummert}.

\begin{definition}\label{def:forcing-rel}
By recursion over the height of formulas, we declare that the forcing relation $p\Vdash\psi$ between a condition $p\in\mathbb P$ and an $\mathcal L_1^{\mathcal M}(\mathsf G)$-sentence~$\psi$ holds precisely when one of the following applies:
\begin{enumerate}[label=(\roman*)]
\item $\psi$ is an atom of~$\mathcal L^{\mathcal M}_1$ and true in~$\mathcal M$,
\item $\psi=\mathsf Gt$ and $\mathcal M\vDash t<|\mathbb V(p)|\land\mathbb V(p)_t=1$ (where $\mathbb V(p)_t$ is the $t$-th entry of the coded sequence~$\mathbb V(p)$),
\item $\psi=\psi_0\lor\psi_1$ and $p\Vdash\psi_i$ for some~$i<2$,
\item $\psi=\exists x<t\,\varphi(x)$ and $p\Vdash\varphi(n)$ for some~$n\in M$ with $\mathcal M\vDash n<t$,
\item $\psi=\exists x\,\varphi(x)$ and $p\Vdash\varphi(n)$ for some~$n\in M$,
\item $\psi=\neg\varphi$ and $q\nVdash\varphi$ for all~$q\preceq p$.
\end{enumerate}
One says that $p$ decides~$\psi$ if we have $p\Vdash\psi$ or $p\Vdash\neg\psi$.
\end{definition}

It may be instructive to observe that the forcing relation has a certain preference for negative information. Indeed, a condition forces $\exists x\,\psi(x)$ only when a witness `forces' it to do so, while it forces $\neg\exists x\,\psi(x)$ whenever this is compatible with the following monotonicity property.

\begin{lemma}\label{lem:forcing-monotone}
Given $p\Vdash\psi$ and $q\preceq p$, we get $q\Vdash\psi$.
\end{lemma}
\begin{proof}
One argues by induction over the height of~$\psi$. For $\psi=\mathsf Gt$, the claim follows from property~(\ref{eq:force-prop1}) of the forcing relation. When we have $p\Vdash\neg\psi$ and $r\preceq q$, we get $r\preceq p$ and hence $r\nVdash\psi$, as needed for~$q\Vdash\neg\psi$. In the other cases, the induction step is immediate.
\end{proof}

The set of sentences that are forced by a fixed condition is not closed under logical equivalence (e.\,g., because the condition will not force all tautologies~$\mathsf Gt\lor\neg\mathsf Gt$). We now show that this is remedied when we consider suitable sets of conditions. Recall that a non-empty set $F\subseteq\mathbb P$ is a filter if $p\succeq q\in F$ entails $p\in F$ and every two elements $p,q\in F$ admit a common bound $r\preceq p,q$ with $r\in F$. We do not demand $F\neq\mathbb P$, though this holds in non-trivial cases.

\begin{definition}
A filter~$G\subseteq\mathbb P$ is generic if each $\mathcal L^{\mathcal M}_1(\mathsf G)$-sentence is decided by some condition in~$G$ and each~$n\in M$ admits a $p\in G$ with $\mathcal M\vDash|\mathbb V(p)|\geq n$.
\end{definition}

The notation is justified since~$G$ will provide a meaningful interpretation for~$\mathsf G$ (but note the different font). As one can see from~\cite{dzhafarov-mummert}, it suffices for our application that $G$ decides all $\Sigma^0_2$-sentences, but we see no reason to think about formula complexity at this point.

\begin{proposition}\label{prop:generics-exist}
Given that the model~$\mathcal M$ is countable, each condition~$p\in\mathbb P$ admits a generic filter~$G\subseteq\mathbb P$ with $p\in G$.
\end{proposition}
\begin{proof}
Fix an enumeration $n_0,n_1,\ldots$ of~$M$ (not in increasing order) and an enumeration $\psi_0,\psi_1,\ldots$ of all $\mathcal L^{\mathcal M}_1(\mathsf G)$-sentences. We set~$p_0=p$ and recursively assume that~$p_i$ is given. If possible, we pick a condition~$p_i'\preceq p_i$ that forces~$\psi_i$. If not, we set $p_i'=p_i$ and note that this condition forces~$\neg\psi_i$ (by clause~(vi) of Definition~\ref{def:forcing-rel}). So in either case, $p_i'\preceq p_i$ decides~$\psi_i$. Due to property~(\ref{eq:force-prop2}) of the forcing relation, we may pick a condition~$p_{i+1}\preceq p_i'$ with $\mathcal M\vDash|\mathbb V(p_{i+1})|\geq n_i$. Now the set
\begin{equation*}
G=\{q\in\mathbb P\,|\,p_i\preceq q\text{ for some }i\in\mathbb N\}
\end{equation*}
has the desired properties, as one readily verifies.
\end{proof}

In the view of the authors, the argument is particularly transparent if we first define the forcing extension as a first-order structure (somewhat deviating from~\cite{dzhafarov-mummert}).

\begin{definition}\label{def:G-bar}
For a generic filter~$G\subseteq\mathbb P$, we set
\begin{equation*}
\bar G=\{n\in M\,|\,\text{there is a $p\in G$ with $\mathcal M\vDash n<|\mathbb V(p)|\land\mathbb V(p)_n=1$}\}.
\end{equation*}
We then define $\mathcal M_G$ as the $\mathcal L^{\mathcal M}_1(\mathsf G)$-structure that extends the $\mathcal L^{\mathcal M}_1$-structure~$\mathcal M$ by interpreting~$\mathsf G$ as $\bar G$.
\end{definition}

We will need the following equivalent characterization.

\begin{lemma}\label{lem:G-bar}
When~$G\subseteq\mathbb P$ is a generic filter, we have
\begin{equation*}
\bar G=\{n\in M\,|\,\text{all $p\in G$ validate $\mathcal M\vDash n<|\mathbb V(p)|\to\mathbb V(p)_n=1$}\}.
\end{equation*}
\end{lemma}
\begin{proof}
To see that the right set is included in the left, it suffices to note that any~$n\in M$ admits a $p\in G$ with $\mathcal M\vDash n<|\mathbb V(p)|$, due to the assumption that~$G$ is generic. For the converse inclusion, assume that $p\in G$ witnesses $n\in\bar G$ according to Definition~\ref{def:G-bar}. Given an arbitrary~$q\in G$ with $\mathcal M\vDash n<|\mathbb V(q)|$, we need to show that we have $\mathcal M\vDash\mathbb V(q)_n=1$. As~$G$ is a filter, there is an $r\in G$ with $r\preceq p,q$. Due to~(\ref{eq:force-prop1}), we get $\mathbb V(p),\mathbb V(q)\sqsubseteq\mathbb V(r)$  and thus $\mathbb V(q)_n=\mathbb V(r)_n=\mathbb V(p)_n=1$ in~$\mathcal M$.
\end{proof}

The following is at the heart of the forcing contruction.

\begin{proposition}\label{prop:all-forced}
For any generic filter~$G\subseteq\mathbb P$ and $\mathcal L^{\mathcal M}_1(\mathsf G)$-sentence~$\psi$, we have
\begin{equation*}
\mathcal M_G\vDash\psi\quad\Leftrightarrow\quad\text{there is a $p\in G$ with $p\Vdash\psi$}.
\end{equation*}
\end{proposition}
\begin{proof}
We use induction over the height of~$\psi$. For $\psi=\mathsf Gt$, we have $\mathcal M_G\vDash\psi$~precisely if~$t^{\mathcal M}$ lies in~$\bar G$, i.\,e., if there is a $p\in G$ with $\mathcal M\vDash t<|\mathbb V(p)|\land\mathbb V(p)_t=1$. The latter is equivalent to~$p\Vdash\psi$ by definition of the forcing relation. In the only other interesting case, we have~$\psi=\neg\varphi$. First assume~$\mathcal M_G\vDash\neg\varphi$. By induction hypothesis, we have $q\not\Vdash\varphi$ for all~$q\in G$. Since~$G$ is generic, we find a $p\in G$ that decides~$\varphi$. But then we can only have $p\Vdash\neg\varphi$. For the converse, we assume~$\mathcal M_G\nvDash\neg\varphi$. By induction hypothesis, we find a~$q\in G$ with $q\Vdash\varphi$. Towards a contradiction, assume that $p\Vdash\neg\varphi$ holds for some~$p\in G$. As $G$ is a filter, we find an $r\in G$ that validates $r\leq p,q$. By Lemma~\ref{lem:forcing-monotone} we get $r\Vdash\varphi$, which contradicts~$p\Vdash\neg\varphi$.
\end{proof}

A central property of the forcing relation is that it can be defined in the ground model. We need this fact in the following form. For the elementary but somewhat intricate proof, we refer the reader to~\cite{dzhafarov-mummert}, where the result is shown as~Lemma~7.7.4.

\begin{lemma}\label{lem:forcing-definable}
For any bounded $\mathcal L^{\mathcal M}_1(\mathsf G)$-formula~$\theta(\mathbf x)$ with a tuple~$\mathbf x$ of free variables, there is a bounded $\mathcal L^{\mathcal M}_1$-formula~$\theta'(\mathbf x,z)$ such that all $\mathbf n$ from~$M$ validate:
\begin{enumerate}[label=(\roman*)]
\item If we have $p\Vdash\theta(\mathbf n)$, there is a $q\preceq p$ with $\mathcal M\vDash\theta'(\mathbf n,\mathbb V(q))$.
\item Given $\mathcal M\vDash\theta'(\mathbf n,\mathbb V(p))$ with $p\in\mathbb P$, we get $p\Vdash\theta(\mathbf n)$.
\item We have $\mathcal M\vDash\theta'(\mathbf n,\sigma)\land\sigma\sqsubseteq\tau\in 2^{<\omega}\to\theta'(\mathbf n,\tau)$.
\end{enumerate}
\end{lemma}

Crucially, we can derive that induction remains valid in generic extensions. For later reference, we present the following proof with more details than in~\cite{dzhafarov-mummert}.

\begin{proposition}\label{prop:force-induction}
When~$G\subseteq\mathbb P$ is a generic filter, $\mathcal M_G$ validates induction along~$\mathbb N$ whenever the induction formula is a $\Sigma_1$-formula of the language~$\mathcal L_1^{\mathcal M}(\mathsf G)$.
\end{proposition}
\begin{proof}
Consider a $\Sigma_1$-formula~$\psi$ of $\mathcal L_1^{\mathcal M}(\mathsf G)$ such that we have
\begin{equation*}
\mathcal M_G\vDash\psi(0)\land\forall x\big(\psi(x)\to\psi(x+1)\big).
\end{equation*}
We write $\psi(x)=\exists y\,\theta(x,y)$ for a bounded formula~$\theta$. Let $\theta'(x,y,z)$ be a bounded $\mathcal L^{\mathcal M}_1$-formula as provided by the previous lemma. By Proposition~\ref{prop:all-forced} (in conjunction with Lemma~\ref{lem:forcing-monotone} and the fact that any two elements of the filter~$G$ have a common bound), we have a~$q=(n,U)\in G$ that validates
\begin{equation*}
q\Vdash\psi(0)\quad\text{and}\quad q\Vdash\neg\exists x\neg\big(\neg\psi(x)\lor\psi(x+1)\big).
\end{equation*}
For each $i\in M$, we consider
\begin{equation*}
U_i=\{\sigma\in U\,|\,\mathcal M\vDash\forall y\leq|\sigma|:\neg\theta'(i,y,\sigma)\}\in\mathcal M,
\end{equation*}
which is a tree by part~(iii) of the lemma above. Let us show that we have
\begin{equation*}
\text{$U_i$ is $\mathcal M$-infinite}\quad\Leftrightarrow\quad\text{$p\Vdash\neg\psi(i)$ for some~$p\preceq q$}.
\end{equation*}
Assuming the right side, note that any $k\in M$ admits a~$p'\preceq p$ with $\mathcal M\vDash|\mathbb V(p')|\geq k$, due to property~(\ref{eq:force-prop2}) of the forcing notion. For any~$j\in M$, we get $p'\not\Vdash\theta(i,j)$, so that part~(ii) of the lemma yields~$\mathcal M\vDash\neg\theta'(i,j,\mathbb V(p'))$. We have thus seen that $U_i$ contains sequences~$\mathbb V(p')$ of arbitrary length. Conversely, if $U_i$ is $\mathcal M$-infinite, then we may consider $p:=(n,U_i)\preceq q$. Aiming at a contradiction, we assume~$p\not\Vdash\neg\psi(i)$. This yields a~$p'\preceq p$ with $p'\Vdash\theta(i,j)$ for some~$j\in M$. By part~(i) of the lemma, we may assume~$\mathcal M\vDash\theta'(i,j,\mathbb V(p'))$, possibly for a different~$p'\preceq p$. Changing the latter again, we may also assume~$\mathcal M\vDash|\mathbb V(p')|\geq j$, due to properties (\ref{eq:force-prop1}) and~(\ref{eq:force-prop2}) together with part~(iii) of the lemma. These properties yield~$\mathbb V(p')\notin U_i$. However, writing~$p'=(n',U')$, we also see that $p'\preceq p$ yields~$\mathbb V(p')\in U'\subseteq U_i$.

We now use induction in~$\mathcal M$ to show that all~$U_i$ are $\mathcal M$-finite. Let us note, again, that the latter is a $\Sigma^0_1$-property for subtrees of~$2^{<\omega}$. For~$i=0$, the claim holds by the equivalence above, since any~$p\preceq q$ validates~$p\Vdash\psi(0)$ and hence~$p\not\Vdash\neg\psi(0)$. To establish the induction step, we assume that $U_{i+1}$ is $\mathcal M$-infinite. This yields a condition $p\preceq q$ with $p\Vdash\neg\psi(i+1)$. We must also have
\begin{equation*}
p\not\Vdash\neg\big(\neg\psi(i)\lor\psi(i+1)\big).
\end{equation*}
Thus there is a $p'\preceq p$ that, in view of~$p'\not\Vdash\psi(i+1)$, must validate~$p'\Vdash\neg\psi(i)$. But then already~$U_i$ is $\mathcal M$-infinite.

For arbitrary~$i\in M$, we now learn that~$p\not\Vdash\neg\psi(i)$ holds for all~$p\preceq q$. But~$\psi(i)$ must be decided by some~$p$ from~$G$, as the latter is generic. We can assume~$p\preceq q$, as any two elements in the filter~$G$ have a common bound. So we must have $p\Vdash\psi(i)$. By Proposition~\ref{prop:all-forced} we get $\mathcal M_G\vDash\psi(i)$, as needed to establish induction.
\end{proof}

We now turn $\mathcal M_G$ into a second-order structure.

\begin{definition}
Given a generic filter~$G\subseteq\mathbb P$, let $\mathcal S_G$ consist of those subsets of~$M$ that are $\Delta_1$-definable in the $\mathcal L_1^{\mathcal M}(\mathsf G)$-structure~$\mathcal M_G$. We define $\mathcal M[G]$ as the second-order structure with the same first-order part as~$\mathcal M$ and second-order part~$\mathcal S_G$.
\end{definition}

As in the usual proof that models of~$\mathsf{I\Sigma}_1$ can be extended into models of~$\mathsf{RCA}_0$ (see, e.\,g., Lemma~IX.1.8 of~\cite{simpson09}), Proposition~\ref{prop:force-induction} entails the following.

\begin{corollary}
We have $\mathcal M[G]\vDash\mathsf{RCA}_0$ for any generic filter~$G\subseteq\mathbb P$.
\end{corollary}

Since any set of~$\mathcal M$ is the canonical interpretation of a predicate symbol in~$\mathcal L_1^{\mathcal M}$, we see that $\mathcal M[G]$ is an $\omega$-extension of~$\mathcal M$. In particular, our fixed tree~$T$ is a set of~$\mathcal M[G]$. Let us verify the corresponding instance of weak K\H{o}nig's lemma.

\begin{lemma}
For any generic filter~$G\subseteq\mathbb P$, we have
\begin{equation*}
\mathcal M[G]\vDash\text{``\,$T$ has an infinite path"}.
\end{equation*}
\end{lemma}
\begin{proof}
In $\mathcal M_G$, the predicate~$\mathsf G$ of~$\mathcal L_1^{\mathcal M}(\mathsf G)$ is interpreted by~$\bar G$, so that the latter is a set of~$\mathcal M[G]$. We show that the characteristic function $f:\mathbb N\to\{0,1\}$ of~$\bar G$ is the desired path, i.\,e., that $\mathcal M[G]\vDash f[n]\in T$ holds for all~$n\in M$. Given~$n\in M$, pick a $p\in G$ with $\mathcal M\vDash|\mathbb V(p)|\geq n$, which is possible as~$G$ is generic. Since we always have $\mathbb V(p)\in T$, it suffices to note that we have $\mathcal M\vDash f(i)=\mathbb V(p)_i$ for every~$i<^{\mathcal M} n$, which holds by Definition~\ref{def:G-bar} and Lemma~\ref{lem:G-bar}.
\end{proof}

We now establish the following additional property.

\begin{proposition}
The extension of~$\mathcal M$ into~$\mathcal M[G]$ is hyperimmune-free, for any generic filter~$G\subseteq\mathbb P$.
\end{proposition}
\begin{proof}
Assume $f$ is a function of~$\mathcal M[G]$. By the construction of the latter, there is a $\Sigma_1$-formula~$\psi$ of $\mathcal L_1^{\mathcal M}(\mathsf G)$ such that we have
\begin{equation*}
\mathcal M[G]\vDash f(m)=n\quad\Leftrightarrow\quad\mathcal M_G\vDash\psi(m,n).
\end{equation*}
In particular, we have $\mathcal M_G\vDash\forall x\exists y\,\psi(x,y)$. From Proposition~\ref{prop:all-forced}, we know that there must be a condition~$q=(k,U)\in G$ that validates
\begin{equation*}
q\Vdash\neg\exists x\neg\exists y\,\psi(x,y).
\end{equation*}
Write $\psi(x,y)=\exists z\,\theta(x,y,z)$ for bounded~$\theta$ and consider a bounded~$\mathcal L_1^{\mathcal M}$-formula~$\theta'$ as provided by Lemma~\ref{lem:forcing-definable}. For~$m\in M$, we put
\begin{equation*}
U_m=\{\sigma\in U\,|\,\mathcal M\vDash\forall y,z\leq|\sigma|:\neg\theta'(m,y,z,\sigma)\}\in\mathcal M.
\end{equation*}
If~$U_m$ was $\mathcal M$-infinite, we would find a $p\preceq q$ with $p\Vdash\neg\exists y\,\psi(m,y)$, just like in the proof of Proposition~\ref{prop:force-induction}. But this would contradict the assumption on~$q$. Invoking unbounded search, we get a total function~$g$ of~$\mathcal M$ with $\mathcal M\vDash k\leq g(m)$ and
\begin{equation*}
\mathcal M\vDash\sigma\in U_m\to|\sigma|<g(m).
\end{equation*}
To establish $\mathcal M[G]\vDash f(m)\leq g(m)$, pick a~$p=(k',U')\in G$ with $\mathcal M\vDash k'\geq g(m)$. We may assume~$p\preceq q$, as any two elements of the filter~$G$ have a common bound. For $p'=(g(m),U')$ we have $p\preceq p'\preceq q$ and thus~$p'\in G$. Due to $\mathcal M\vDash|\mathbb V(p')|=g(m)$ we get~$\mathbb V(p')\notin U_m$, so that we find $n,i\in M$ with
\begin{equation*}
\mathcal M\vDash n,i\leq g(m)\land\theta'(m,n,i,\mathbb V(p')).
\end{equation*}
Now part~(ii) of Lemma~\ref{lem:forcing-definable} yields $p'\Vdash\psi(m,n)$. We can thus use Proposition~\ref{prop:all-forced} to infer~$\mathcal M_G\vDash\psi(m,n)$ and then $\mathcal M[G]\vDash f(m)=n\leq g(m)$.
\end{proof}

Finally, we derive a theorem that combines Harrington's famous conservation result with the hyperimmune-free basis theorem of C.~Jockusch and R.~Soare~\cite{jockusch-soare}.

\begin{theorem}\label{thm:hypfree-extend}
Any countable model of~$\mathsf{RCA}_0$ has a hyperimmune-free $\omega$-extension that validates~$\mathsf{WKL}_0$.
\end{theorem}
\begin{proof}
Write~$\mathcal M_0$ for the given model of~$\mathsf{RCA}_0$. We recursively assume that countable $\omega$-extensions~$\mathcal M_0\subseteq\ldots\subseteq\mathcal M_n$ with $\mathcal M_i\vDash\mathsf{RCA}_0$ have already been constructed. Once and for all (i.\,e., independently of~$n$), pick enumerations of all~$T_{ij}\in\mathcal M_i$ with
\begin{equation*}
\mathcal M_i\vDash\text{``$\,T_{ij}\subseteq 2^{<\omega}$ is an infinite tree"}.
\end{equation*}
Note that being an infinite tree is arithmetical, so that satisfaction is essentially independent of~$\mathcal M_i$. For the recursion step, let $n$ be the Cantor code of the pair~$(i,j)$, where we have~$i\leq n$. We use the previous results of this section with $\mathcal M_n$ and $T_{ij}$ at the place of~$\mathcal M$ and~$T$. This yields a hyperimmune-free $\omega$-extension of~$\mathcal M_n$ into a model $\mathcal M_{n+1}=\mathcal M_n[G]$ such that we have
\begin{equation*}
\mathcal M_{n+1}\vDash\mathsf{RCA}_0+\text{``$T_{ij}$ has an infinite path"}.
\end{equation*}
Having completed the recursion, we put~$\mathcal M_\omega=\bigcup_{n\in\omega}\mathcal M_n$, i.\,e., we take the union of the second-order parts and declare that $\mathcal M_\omega$ has the same first-order part as the models~$\mathcal M_n$. It is straightforward to see that $\mathcal M_\omega$ satisfied~$\mathsf{WKL}_0$ (as the relevant statements are~$\Pi^1_2$ and hence preserved under chains of~$\omega$-extensions). We inductively learn that each $\mathcal M_n$ is a hyperimmune-free extension of~$\mathcal M_0$, which~entails that the same holds for~$\mathcal M_\omega$.
\end{proof}

\section{Conservativity for well-ordering principles}\label{sect:wo-princ}

In this section, we prove that $\mathsf{WKL}_0$ is conservative over~$\mathsf{RCA}_0$ for well-ordering principles. We also show that conservativity fails when we replace linear orders by the somewhat more general well-quasi-orders. We note that there is a substantial body of results on the reverse mathematics of well-ordering principles, which was briefly reviewed in the introduction.

To be specific, we declare that a partial order is given as a set~$X\subseteq\mathbb N$ of coded pairs such that the relation~${\leq_X}\subseteq\mathbb N^2$ with $m\leq_X n$ for~$(m,n)\in X$ is transitive and antisymmetric with field~$\{n\in\mathbb N\,|\,n\leq_X n\}$. For the following result we stress that well-orders are, in particular, linear. As mentioned in the introduction, a pointwise version of the result was given by Kreuzer and Yokoyama (see Theorem~3.2 of \cite{kreuzer-yokoyama}).

\begin{theorem}\label{thm:cons-wo-princ}
If $\mathsf{WKL}_0$ proves a statement of the form
\begin{equation*}
\forall X\big(\text{``$X$ is a well-order"}\to\psi(X)\big)\quad\text{with}\quad\psi\in\Pi^1_1,
\end{equation*}
then $\mathsf{RCA}_0$ proves the same statement.
\end{theorem}
\begin{proof}
We derive that our statement holds in all countable models~$\mathcal M\vDash\mathsf{RCA}_0$, which suffices due to completeness and the downward L\"owenheim-Skolem theorem (see, e.\,g., Section~5.1 of~\cite{dzhafarov-mummert} for a discussion of these results in the context of second-order arithmetic). By Theorem~\ref{thm:hypfree-extend}, we have a hyperimmune-free $\omega$-extension of~$\mathcal M$ into a model~$\mathcal N\vDash\mathsf{WKL}_0$. Consider an~$X\in\mathcal M$ with
\begin{equation*}
\mathcal M\vDash\text{``$X$ is a linear order"}\land\neg\psi(X).
\end{equation*}
The same statement holds in~$\mathcal N$, as it is~$\Sigma^1_1$. Given that $\mathsf{WKL}_0$ proves the statement from the theorem, we have a function~$f$ of~$\mathcal N$ such that $\mathcal N\vDash f(n+1)<_X f(n)$ holds for all~$n\in M$, where we write~$M$ for the shared first-order part of our models. Let us consider a function~$g$ of~$\mathcal M$ that dominates~$f$, which exists since the extension is hyperimmune-free. By primitive recursion, we get a function $h$ of~$\mathcal M$ that has start value $h(0)=f(0)$ and satisfies
\begin{equation*}
\mathcal M\vDash\text{``$h(n+1)$ is the $\leq_X$-maximal $x\leq_{\mathbb N}g(n+1)$ with $x<_X h(n)$"}.
\end{equation*}
To ensure that there is an~$x$ with the indicated property, we show $\mathcal N\vDash f(n)\leq_X h(n)$ by induction in~$\mathcal N$. In the induction step, we learn that $x=f(n+1)$ satisfies both $x\leq_{\mathbb N} g(n+1)$ and $x<_X f(n)\leq_X h(n)$, so that we get $f(n+1)\leq_X h(n+1)$ by maximality (all in the sense of our models). Now~$h$ witnesses that $X$ is ill-founded according to~$\mathcal M$. Hence the latter satisfies the statement from the theorem.
\end{proof}

As a corollary to the proof, we record the following conservativity result for instances of the ascending descending sequence principle (introduced in~\cite{hirschfeldt-shore}).

\begin{corollary}\label{cor:ads}
    If $\mathsf{WKL}_0$ proves a statement of the form
    \begin{equation*}
        \forall X\big(\varphi(X)\land\text{``$X$ is a linear order"}\to\text{``there is a strictly monotone $f:\mathbb N\to X$"}\big)
    \end{equation*}
    with $\varphi\in\Sigma^1_1$, then $\mathsf{RCA}_0$ proves the same statement.
\end{corollary}
\begin{proof}
The previous proof with $\varphi$ at the place of~$\neg\psi$ covers the case where~$f$ is descending. It is straightforward to adapt the argument to the ascending case.
\end{proof}

Theorem~\ref{thm:cons-wo-princ} becomes false when we drop linearity, as the statement
\begin{equation*}
\forall X\big(\text{``$X$ is well-founded"}\to\text{``$X$ is no infinite subtree of~$2^{<\omega}$\,"}\big)
\end{equation*}
is a reformulation of weak K\H{o}nig's lemma (with the partial order~$\sqsupseteq$ on~$X$). To prove a stronger result, we recall that~$X$ is a well-partial-order if any infinite~sequence~$x_0,x_1,\ldots\subseteq X$ admits $i<j$ with $x_i\leq_X x_j$ (see Remark~\ref{rmk:equiv-wpo} about alternative definitions over~$\mathsf{RCA}_0$).

\begin{theorem}\label{thm:WKL-wpo}
Over~$\mathsf{RCA}_0$, weak K\H{o}nig's lemma is equivalent to
\begin{equation*}
\forall X\big(\text{``$X$ is a well-partial-order"}\to\text{``\,$W(X)$ is a well-partial-order"}\big)
\end{equation*}
for some computable transformation~$W$ (which is given as the code of a program that decides~$m\leq_{W(X)}n$ with oracle~$X$).
\end{theorem}
\begin{proof}
We use the result that weak K\H{o}nig's lemma is equivalent to $\Sigma^0_1$-separation (see Lemma~IV.4.4 of~\cite{simpson09}). For two disjoint $\Sigma^0_1$-collections $\mathcal Y_i=\{m\,|\,\exists n\,\theta_i(m,n)\}$, this principle asserts that there is a set~$S$ with~$\mathcal Y_0\subseteq S$ and~$\mathcal Y_1\subseteq\mathbb N\backslash S$. By coding the data into a single set~$Y=\{\langle i,m,n\rangle\,|\,\theta_i(m,n)\}$, we reach the equivalent principle that any set $Y$ with
\begin{equation}\label{eq:sep-prem}
\langle 0,m,n\rangle\in Y\land\langle 1,m',n'\rangle\in Y\to m\neq m'
\end{equation}
admits a set $S$ such that we have
\begin{equation}\label{eq:sep-concl}
\langle 0,m,n\rangle\in Y\to m\in S\quad\text{and}\quad\langle 1,m,n\rangle\in Y\to m\notin S.
\end{equation}
Let $\mathsf S(Y)$ be the partial order with underlying set~$2^{<\omega}$ such that $\sigma<_{\mathsf S(Y)}\tau$ holds precisely if we have $|\sigma|<|\tau|$ and there are $m<|\sigma|$ and $n<|\tau|$ with $\langle 0,m,n\rangle\in Y$ and $\sigma_m=0$ or with $\langle 1,m,n\rangle\in Y$ and $\sigma_m=1$.

\begin{claim}
There is an $S$ that validates~(\ref{eq:sep-concl}) precisely if~$\mathsf S(Y)$ is no well-partial-order.
\end{claim}

If~(\ref{eq:sep-concl}) holds, the sequence $S[0],S[1],\ldots$ admits no $m<n$ with $S[m]\leq_{\mathsf S(Y)} S[n]$. For the converse, assume that we have an infinite sequence~$\sigma^0,\sigma^1,\ldots\subseteq 2^{<\omega}$ such that $\sigma^m\not\leq_{\mathsf S(Y)}\sigma^n$ holds for all~$m<n$. We may assume $0<|\sigma^0|<|\sigma^1|<\ldots$ by passing to a subsequence. Let us put~$S=\{m\in\mathbb N\,|\,\sigma^m_m=1\}$ (where $\sigma^i_j$ is the $j$-th entry of~$\sigma^i$). If we have~$\langle 0,m,n\rangle\in Y$, then $\sigma^m\not\leq_{\mathsf S(Y)}\sigma^N$ for~$N=\max(m+1,n)$ forces~$\sigma^m_m=1$. Thus we get~$m\in S$, as required for~(\ref{eq:sep-concl}). Similarly, $\langle 1,m,n\rangle\in Y$ implies~$m\notin S$.

To define~$W(X)$, we write~$Y$ for the underlying set of~$X$. Let~$e_Y:Y\to 2^{<\omega}$ be a canonical injection that is bijective in case~$Y$ is infinite (compose $Y\cong I\subseteq\mathbb N\cong 2^{<\omega}$ for an initial segment~$I$ of~$\mathbb N$). Let us consider the initial segment~$J\subseteq\mathbb N$ such that we have~$N\in J$ precisely if the following holds:
\begin{enumerate}[label=(\roman*)]
\item for all~$m,n\in Y\cap\{0,\ldots,N-1\}$ we have
\begin{equation*}
m\leq_X n\quad\Leftrightarrow\quad e_Y(m)\leq_{\mathsf S(Y)}e_Y(n),
\end{equation*}
\item condition~(\ref{eq:sep-prem}) holds for all~$m,m',n,n'<N$.
\end{enumerate}
We now define~$W(X)$ as the linear order on~$Y$ that inverts $\leq_{\mathbb N}$ on~$J$ and keeps it unchanged on~$Y\backslash J$, i.\,e., with
\begin{equation}\label{eq:W(X)}
m\leq_{W(X)}n\quad\Leftrightarrow\quad\begin{cases}
m\geq_{\mathbb N} n & \text{if $m,n\in J\cap Y$},\\
m\leq_{\mathbb N} n & \text{if $m,n\in Y\backslash J$},\\
0=0 & \text{if $m\in J\cap Y$ and $n\in Y\backslash J$}.
\end{cases}
\end{equation}
Let us show that $\Sigma^0_1$-separation is equivalent to the implication from the theorem. To establish the latter, assume that~$W(X)$ is no well-partial-order. Then~$J\cap Y$ must be infinite. It follows that $e_Y$ is an order isomorphism~$X\cong\mathsf S(Y)$ and that~(\ref{eq:sep-prem}) is always satisfied. Using $\Sigma^0_1$-separation, we can validate~(\ref{eq:sep-concl}). Now the claim above tells us that~$X$ is no well-partial-order either. Conversely, we derive separation for a set~$Y$ that satisfies~(\ref{eq:sep-prem}). We may assume that~$Y$ is infinite (e.\,g., extend~$Y$ by~all tuples $\langle 0,m,n'\rangle$ such that we have $\langle 0,m,n\rangle\in Y$ for some~$n<n'$). Let~$X$ be the order with underlying set~$Y$ such that~$e_Y$ is an isomorphism~$X\cong\mathsf S(Y)$. We then have~$J=\mathbb N$, so that~$W(X)$ is no well-partial-order. By the implication from the theorem, we learn that~$\mathsf S(Y)$ is no well-partial-order either. Finally, the claim above tells us that~(\ref{eq:sep-concl}) can be satisfied.
\end{proof}

As mentioned in the introduction, many important principles from reverse mathematics have been characterized by well-ordering principles, i.\,e., by transformations of linear orders that preserve well-foundedness. The literature also contains characterizations by transformations of well-partial-orders, which are often even easier to state (though not easier to prove). As an example, $\mathsf{RCA}_0$ proves that arithmetical comprehension is equivalent to Higman's lemma~\cite{simpson-higman}, which asserts that the finite sequences over a well-partial-order can again be equipped with a natural well-partial-ordering. In light of such a result, the transformation~$W$ from the proof of Theorem~\ref{thm:WKL-wpo} seems rather unsatisfactory, not least because~$W(X)$ heavily depends on the underlying set of~$X$ (and not just on its order type). The second author has indeed shown (see Proposition~5.3.6 of~\cite{uftring-thesis}) that Theorem~\ref{thm:WKL-wpo} cannot hold for any~$W$ that is natural in a certain sense. Specifically, the theorem becomes false when we demand that~$W$ is a dilator on well-partial-orders (see~\cite{frw-kruskal} for this notion and~\cite{girard-pi2} for Girard's original dilators on linear orders).

\begin{remark}\label{rmk:equiv-wpo}
It is known that the following conditions on a partial order~$X$ are equivalent over relatively weak theories but not all over~$\mathsf{RCA}_0$ (see~\cite{cholak-RM-wpo}):
\begin{enumerate}[label=(\roman*)]
\item Any infinite sequence~$f:\mathbb N\to X$ admits~$i<j$ with~$f(i)\leq_X f(j)$.
\item Whenever~$B$ is a subset of the underlying set of~$X$, there is a finite~$B_0\subseteq B$ such that each~$x\in B$ admits an~$x_0\in B_0$ with~$x_0\leq_X x$.
\item There is no infinitely descending sequence and no infinite antichain in~$X$.
\item Every linear extension of~$X$ is a well-order.
\item For any~$f:\mathbb N\to X$, there is an infinite set~$P\subseteq\mathbb N$ such that~$f(i)\leq_X f(j)$ holds for all~$i<j$ in $P$.
\end{enumerate}
We have taken~(i) as our definition of well-partial-orders. Let us now show:
\begin{enumerate}[label=(\alph*)]
\item Theorem~\ref{thm:WKL-wpo} remains valid when we adopt~(ii), (iii) or~(iv) as the definition of well-partial-orders.
\item Our theorem also remains valid when we adopt~(v) and extend the base theory~$\mathsf{RCA}_0$ by the infinite pigeonhole principle (which holds in $\omega$-models).
\end{enumerate}
Over~$\mathsf{RCA}_0$, condition~(i) is equivalent to~(ii) and implies~(iii) and~(iv). To obtain~(a), we show that $\mathsf{RCA}_0$ proves the converse implications for the orders $\mathsf S(Y)$ and~$W(X)$ from the previous proof (while a stronger base theory is needed for general orders). For orders of the form~$W(X)$, this is true since these orders are linear. Let us now assume that~(i) fails for an order of the form~$\mathsf S(Y)$. We derive that~(iii) and~(iv) fail as well. Given $\sigma^0,\sigma^1,\ldots$ with $\sigma^i\not\leq_{\mathsf S(Y)}\sigma^j$ for~$i<j$, we may assume that we have $|\sigma^0|<|\sigma^1|<\ldots$, as in the previous proof. This immediately yields $\sigma^i\not\geq_{\mathsf S(Y)}\sigma^j$ for~$i<j$, so that we have an antichain that violates~(iii). We can also conclude that $\Sigma=\{\sigma^i\,|\,i\in\mathbb N\}$ exists as a set. Crucially, we get $\sigma^i\not<_{\mathsf S(Y)}\tau$ for any~$\tau$ (since otherwise $\sigma^i<_{\mathsf S(Y)}\sigma^j$ when~$|\sigma^j|\geq|\tau|$). Let~$\prec_0$ be any linear extension of~$<_{\mathsf S(Y)}$ on the set~$\mathsf S(Y)\backslash\Sigma$ (see Observation~6.1 of~\cite{downey98} for a construction in~$\mathsf{RCA}_0$). To~get a linearization on all of~$\mathsf S(Y)$, we can now set
\begin{equation*}
    \rho\prec\tau\quad\Leftrightarrow\quad\begin{cases}
        \text{either }\rho,\tau\in \mathsf S(Y)\backslash\Sigma\text{ and }\rho\prec_0\tau,\\
        \text{or }\rho\in \mathsf S(Y)\backslash\Sigma\text{ and }\tau\in\Sigma,\\
        \text{or }\rho=\sigma^j\text{ and }\tau=\sigma^i\text{ with }i<j.
    \end{cases}
\end{equation*}
Clearly, the $\sigma^i$ witness that~$\prec$ is ill-founded, so that~(iv) fails. To establish~(b), we need to show that~(i) implies~(v) for the orders~$\mathsf S(Y)$ and~$W(X)$, using the pigeonhole principle. The latter ensures that~(v) holds for any sequence with finite range. Given a sequence~$n_0,n_1,\ldots\subseteq W(X)$, we may thus assume that $n_i<_{\mathbb N}n_{i+1}$ holds for all~$i\in\mathbb N$. If~(v) fails, the set~$J$ from~(\ref{eq:W(X)}) must thus be equal to~$\mathbb N$. But then~(i) fails as well. Similarly, it suffices to consider a sequence~$\sigma^0,\sigma^1,\ldots\subseteq \mathsf S(Y)$ with $|\sigma^i|<|\sigma^{i+1}|$ for all~$i\in\mathbb N$. Assuming~(i), we find~$i(0)<j(0)<i(1)<j(1)<\ldots$ such that all~$k\in\mathbb N$ validate~$\sigma^{i(k)}<_{\mathsf S(Y)}\sigma^{j(k)}$. We get~$\sigma^{i(k)}<_{\mathsf S(Y)}\sigma^{i(k+1)}$ due to the definition of~$\mathsf S(Y)$. This yields the infinitely ascending sequence required by~(v).
\end{remark}

\section{Conservativity for isolated existence}\label{sect:isolated}

This section is devoted to a proof and an application of the following theorem. A definition of the isolated existence quantifier~$\exists^i$ can be found in the introduction of the present paper.

\begin{theorem}\label{thm:isolated-existence}
    The theory~$\mathsf{WKL}_0$ is conservative over~$\mathsf{RCA}_0$ for statements of the form~$\forall X(\varphi(X)\to\exists^i Y\,\psi(X,Y))$ with~$\varphi\in\Sigma^1_1$ and arithmetical~$\psi$.
\end{theorem}
\begin{proof}
    Let us assume~$\mathsf{WKL}_0\vdash\forall X(\varphi(X)\to\exists^i Y\,\psi(X,Y))$. We consider a countable model~$\mathcal M\vDash\mathsf{RCA}_0$. For a given~$X\in\mathcal M$ with $\mathcal M\vDash\varphi(X)$, we invoke Corollary~5.16 of~\cite{simpson-tanaka-yamazaki} to find two $\omega$-extensions $\mathcal N_i\vDash\mathsf{WKL}_0$ of~$\mathcal M$ such that
    \begin{enumerate}[label=(\roman*)]
    \item $\mathcal M$ contains all sets that lie in both~$\mathcal N_0$ and~$\mathcal N_1$,
    \item $\mathcal N_0$ and~$\mathcal N_1$ satisfy the same second-order sentences with number and set parameters from~$\mathcal M$.
    \end{enumerate}
    We get $\mathcal N_i\vDash\varphi(X)$ since~$\varphi$ is~$\Sigma^1_1$. Consider a~$Y_0\in\mathcal N_0$ with $\mathcal N_0\vDash\psi(X,Y_0)$ that is isolated by a coded sequence~$\sigma\in 2^{<\omega}$ in the shared first-order part of our models. The latter means that we have
    \begin{equation}\label{eq:N_0-one}
        \mathcal N_0\vDash\sigma\sqsubset Y_0\land\forall Z\big(\sigma\sqsubset Z\land\psi(X,Z)\to Y_0=Z),
    \end{equation}
    where $\sigma\sqsubset Y$ amounts to $\sigma=Y[|\sigma|]$. For $n$ in the first-order part, we get
    \begin{equation}\label{eq:N_0-two}
        n\in Y_0\quad\Leftrightarrow\quad\mathcal N_0\vDash\exists Y(\sigma\sqsubset Y\land\psi(X,Y)\land n\in Y).
    \end{equation}
    In view of~(ii), there must also be a~$Y_1\in\mathcal N_1$ with $\mathcal N_1\vDash\psi(X,Y_1)$ that is isolated by the same~$\sigma$, i.\,e., such that (\ref{eq:N_0-one}) and~(\ref{eq:N_0-two}) hold with index~$1$ at the place of~$0$. The two versions of~(\ref{eq:N_0-two}) together with~(ii) ensure~$Y_0=Y_1$. We thus get~$Y_0\in\mathcal M$ due to~(i). Downward absoluteness yields~$\mathcal M\vDash\psi(X,Y_0)$. Let us note that we obtain
\begin{equation*}
\mathsf{RCA}_0\vdash\forall X\big(\varphi(X)\to\exists Y\,\psi(X,Y)\big)
\end{equation*}
even when~$\psi$ is~$\Pi^1_1$. When it is arithmetical, the statement in~(\ref{eq:N_0-one}) is~$\Pi^1_1$ and hence downward absolute, so that $\exists$ may be strengthened to~$\exists^i$.
\end{proof}

The previous proof is similar to the proof of Theorem~5.17 from~\cite{simpson-tanaka-yamazaki} but adapts it to the case of isolated rather than unique existence. It relies on a forcing construction from~\cite{simpson-tanaka-yamazaki} that is much more difficult than the one in Section~\ref{sect:hyp-free} above. This makes it worthwhile to give a new proof of the following case by our methods.

\begin{proof}[Alternative Proof of Theorem~\ref{thm:isolated-existence} for $\psi\in\Sigma^0_3$]
    By the Kleene normal form theorem, we find a bounded formula~$\theta$ with
    \begin{equation*}
        \mathsf{RCA}_0\vdash\psi(X,Y)\leftrightarrow\exists k\forall m\exists n\,\theta(k,m,X,Y[n]).
    \end{equation*}
    We may assume that $\theta(k,m,X,\sigma)$ and $\sigma\sqsubset\tau\in 2^{<\omega}$ entail~$\theta(k,m,X,\tau)$.
    
Given a countable model~$\mathcal M\vDash\mathsf{RCA}_0$ and some~$X\in\mathcal M$ with $\mathcal M\vDash\varphi(X)$, we use Theorem~\ref{thm:hypfree-extend} to get a hyperimmune-free $\omega$-extension~$\mathcal N\vDash\mathsf{WKL}_0$. Assuming that we have~$\mathsf{WKL}_0\vdash\varphi(X)\to\exists^iY\,\psi(X,Y)$, we pick a~$Y\in\mathcal N$ with $\mathcal N\vDash\psi(X,Y)$ and a sequence~$\sigma$ that isolates it (as in~(\ref{eq:N_0-one}) from the previous proof). For a suitable~$k$ from the shared first-order part, we thus have
    \begin{equation*}
        \mathcal N\vDash\forall Z\big(Y=Z\leftrightarrow\sigma\sqsubset Z\land\forall m\exists n\,\theta(k,m,X,Y[n])\big).
    \end{equation*}
    Within~$\mathcal M$, we define~$T$ as the tree that consists of all sequences~$\langle\tau^0,\ldots,\tau^{m-1}\rangle$ with~$\sigma\sqsubset\tau^0\sqsubset\ldots\sqsubset\tau^{m-1}\in 2^{<\omega}$ such that all~$i<m$ validate the following:
    \begin{enumerate}[label=(\roman*)]
    \item we have $\theta(k,i,X,\tau^i)$,
    \item we have $\neg\theta(k,i,X,\tau)$ when $\tau^{i-1}\sqsubset\tau\sqsubset\tau^i$ (read $\tau^{-1}=\sigma$).
    \end{enumerate}
    In~$\mathcal N$, one obtains a path if one recursively takes the shortest~$\tau^i$ with $\tau^{i-1}\sqsubset\tau^i\sqsubset Y$ that satisfies~$\theta(k,i,X,\tau^i)$. To see that this is the only path, assume we have a sequence $\rho^0,\rho^1,\ldots$ in~$\mathcal N$ such that $\langle\rho^0,\ldots,\rho^{m-1}\rangle\in T$ holds for all~$m\in\mathbb N$. Take~$Z$ such that all~$i\in\mathbb N$ validate~$\rho^i\sqsubset Z$. Then~(i) witnesses~$\forall m\exists n\,\theta(k,m,X,Z[n])$, so that we get~$Y=Z$. In view of~(ii), we can use induction on~$i$ to show that $\rho^i$ coincides with~$\tau^i$ from the path that we have constructed.

    As our $\omega$-extension is hyperimmune-free, we have an~$f\in\mathcal M$ with $\tau^i<_{\mathbb N}f(i)$, where~$i\mapsto\tau^i$ is the path from the previous paragraph. Consider the bounded tree
    \begin{equation*}
        T_0=\big\{\langle\rho^0,\ldots,\rho^{m-1}\rangle\in T\,\big|\,\rho^i<_{\mathbb N}f(i)\text{ for all }i<m\big\}\in\mathcal M.
    \end{equation*}
    Essentially, the result now follows from the fact that unique branches are computable. To give a more explicit argument, we consider the function $g:\mathbb N\to T_0$ in~$\mathcal M$ that traverses~$T$ by depth-first search, moving to the right first. In terms of the Kleene-Brouwer order~$<_{\mathsf{KB}}$, this means that~$g$ is a descending enumeration of an end segment. We can infer that~$g$ does not move to the left of our path, i.\,e., that we have
    \begin{equation}\label{eq:g-right}
        \langle\tau^0,\ldots,\tau^{i-1}\rangle\leq_{\mathsf{KB}}g(i).
    \end{equation}
    The latter remains valid when we pass to a subsequence with~$|g(i)|\geq i$. For~$m\leq i$, let $g(i,m)\sqsubseteq g(i)$ be the initial segment that has length~$|g(i,m)|=m$. We say that $i$~is~$m$-true if no~$\rho\in T_0$ of height~$i$ lies to the left of~$g(m,i)$, i.\,e., if we have
    \begin{equation}\label{eq:left-finite}
        \rho<_{\mathsf{KB}}g(i,m)\text{ and }|\rho|=i\quad\Rightarrow\quad g(i,m)\sqsubset\rho.
    \end{equation}
    Note that this is decidable since~$T_0$ is bounded. To see that each~$m$ admits an~$i$ that is $m$-true, first take an~$i_0$ such that $g(i,m)=g(i_0,m)$ holds for all~$i\geq i_0$. In particular, the bounded tree~$\{\rho\in T_0\,|\,g(i_0,m)\sqsubseteq\rho\}$ is infinite and must thus have a path in~$\mathcal N$. Since the path $i\mapsto\tau^i$ is unique, we thus get $g(i_0,m)=\langle\tau^0,\ldots,\tau^{m-1}\rangle$. It follows that the tree
    \begin{equation*}
        T'=\big\{\rho\in T_0\,\big|\,\rho\sqsubset g(i_0,m)\big\}\cup\big\{\rho\in T_0\,\big|\,\rho<_{\mathsf{KB}}g(i_0,m)\text{ and }g(i_0,m)\not\sqsubset\rho\big\}
    \end{equation*}
    must be finite. Take an~$i\geq i_0$ such that $\rho\in T'$ implies~$|\rho|<i$. It is straightforward to conclude that this validates~(\ref{eq:left-finite}). We now show that we have
    \begin{equation*}
        g(i,m)=\big\langle\tau^0,\ldots,\tau^{m-1}\big\rangle\quad\text{when~$i$ is $m$-true}.
    \end{equation*}
    In view of~(\ref{eq:g-right}), we at least get~$\geq_{\mathsf{KB}}$ at the place of the desired equality. If we had a strict inequality $>_{\mathsf{KB}}$, then~$\rho=\langle\tau^0,\ldots,\tau^{i-1}\rangle$ would validate the premise of~(\ref{eq:left-finite}). Due to the conclusion of the latter, we would get our equality after all. By searching for~$i$ that are~$m$-true, we now learn that the path~$m\mapsto\tau^m$ lies in~$\mathcal M$. We have previously seen~$\tau^0\sqsubset\tau^1\sqsubset\ldots\sqsubset Y$, so that we get~$Y\in\mathcal M$. The statement that~$Y$ is isolated by~$\sigma$ is~$\Pi^1_1$ (as in the previous proof). We can thus conclude~$\mathcal M\vDash\exists^iY\,\psi(X,Y)$, as desired.
\end{proof}

The fact that unique paths in bounded trees are computable is one central ingredient of the previous proof. This yields an interesting contrast with a comment of D.~Hirschfeldt, who mentions the computability of unique paths as an example for ``a result of computable mathematics [that does not have] a reverse mathematical analog" (see Section~4.7 of~\cite{hirschfeldt-slicing}).

In the rest of this section, we discuss an application to continuous functions with isolated singularities. To handle continuous functions $f:D_f\to\mathbb R$ with $D_f\subseteq\mathbb R$ in the context of~$\mathsf{RCA}_0$, we use the standard encoding from Definition~II.6.1 of~\cite{simpson09}. Let us recall that codes are sets~$\Phi\subseteq\mathbb N\times\mathbb Q^4$ of tuples with strictly positive third and fifth component. The idea is that~$f$ is coded by~$\Phi$ if we have $f(B_\delta(a))\subseteq\overline B_\varepsilon(b)$ whenever there is an~$n$ with $(n,a,\delta,b,\varepsilon)\in\Phi$ (where $B_\delta(a)=\{x\,:\,|x-a|<\delta\}$). For any $\Phi$ that satisfies suitable coherence conditions, we put
\begin{multline*}
    D_\Phi=\big\{x\in\mathbb R\,\big|\,\text{each $\varepsilon>0$ in~$\mathbb Q$ admits $n,a,\delta,b$ with}\\
    \text{$(n,a,\delta,b,\varepsilon)\in\Phi$ and $|x-a|<\delta$}\big\}.
\end{multline*}
Here real numbers are given via the standard encoding as fast converging Cauchy sequences (see Definition~II.4.4 of~\cite{simpson09}). Within $\mathsf{RCA}_0$, one shows that each~$x\in D_\Phi$ admits a unique value~$f_\Phi(x)\in\mathbb R$ that is characterized by
\begin{equation*}
    |x-a|<\delta\text{ and }(n,a,\delta,b,\varepsilon)\in\Phi\text{ for some }n,a,\delta\quad\Rightarrow\quad\big|f_\Phi(x)-b\big|\leq\varepsilon.
\end{equation*}
Concerning the use of $<$ and~$\leq$, we recall that strict and weak comparisons in~$\mathbb R$ are $\Sigma^0_1$ and~$\Pi^0_1$, respectively. This means that $f_\Phi(x)=y$ is a $\Pi^0_1$-formula. When we refer to a continuous function $f:D_f\to\mathbb R$ in the sequel, we assume that it is given via a~$\Phi$ with $f=f_\Phi$, though we usually leave~$\Phi$ implicit.

Weak K\H{o}nig's lemma is equivalent to the principle that any continuous function $f:[0,1]\to\mathbb R$ is bounded, over~$\mathsf{RCA}_0$. Simpson, Tanaka and Yamazaki~\cite{simpson-tanaka-yamazaki} are careful to explain why this does not contradict their conservativity result for unique existence, even though one has the supremum as a unique upper bound. In fact, one can already make the point for Harrington's result, as we can demand an upper bound from~$\mathbb Q$ (so that we have a number quantifier). In any case, no contradiction arises, since we get another existential quantifier from the universal premise that every~$x\in[0,1]$ lies in the domain of~$f$. For a suitable class of functions that are guaranteed to be total, on the other hand, Harrington's result does yield the boundedness principle over~$\mathsf{RCA}_0$. In the following, we use Theorem~\ref{thm:isolated-existence} to extend this observation to functions with the following property.

\begin{definition}
Consider a continuous function~$f:D_f\to\mathbb R$. Elements of~$\mathbb R\backslash D_f$ will be called singularities of~$f$. We say that a singularity~$x$ of~$f$ is isolated if there is an~$\varepsilon>0$ such that $|x-y|\in(0,\varepsilon)$ implies~$y\in D_f$.
\end{definition}

We can now make the promised result precise.

\begin{theorem}
For a $\Sigma^1_1$-formula~$\varphi$ with
\begin{equation*}
\mathsf{RCA}_0\vdash\text{``every continuous~$f:D_f\to\mathbb R$ with $\varphi(f)$ has only isolated singularities"},
\end{equation*}
we also obtain
\begin{multline*}
\mathsf{RCA}_0\vdash\text{``every continuous $f:D_f\to\mathbb R$ with $\varphi(f)$ is bounded on}\\
\text{closed intervals that are contained in~$D_f$"}.
\end{multline*}
\end{theorem}
\begin{proof}
Let~$\varphi'(f)$ express that ``$f:D_f\to\mathbb R$ is a continuous function with $\varphi(f)$ that is unbounded on $[a,b]\cap D_f$" (with free variables~$a$ and~$b$). Given weak K\H{o}nig's lemma, $\varphi'(f)$ entails that we have~$[a,b]\not\subseteq D_f$, i.\,e., that $f$ has a singularity in~$[a,b]$. The assumption of the theorem guarantees that all singularities are isolated. We thus get
\begin{equation*}
\mathsf{WKL}_0\vdash\forall f\big(\varphi'(f)\to\text{``$f$ has an isolated singularity in~$[a,b]$"}\big).
\end{equation*}
We cannot apply Theorem~\ref{thm:isolated-existence} immediately, because the same singularity is represented by different Cauchy sequences, so that the representations are not isolated in the required sense. This obstacle can be overcome via binary expansions. For convenience, we assume~$[a,b]=[0,1]$. Each~$Y\subseteq\mathbb N$ determines a real~$r^Y\in[0,1]$ that is given as the fast Cauchy sequence
\begin{equation*}
r^Y=\big(r^Y_n\big)\quad\text{with}\quad r^Y_n=\textstyle\sum_{i<n}\,2^{-i-1}\cdot\chi_Y(i).
\end{equation*}
One proves in~$\mathsf{RCA}_0$ that any real~$x\in[0,1]$ is equal to one of the form~$r^Y$. Indeed, this is straightforward when~$x$ is rational. When it is not, one can decide in which rational intervals~$x$ is contained (since~$<$ and~$\leq$ on~$\mathbb R$ are $\Sigma^0_1$ and $\Pi^0_1$, respectively). This makes it possible to compute the desired~$Y$. Let us observe
\begin{equation*}
\big|r^Y-r^Z\big|>2^{-n}\quad\Rightarrow\quad Y[n]\neq Z[n].
\end{equation*}
To see that $Y\subseteq\mathbb N$ is isolated when the same holds for~$r^Y$, we must also consider the case where~$r^Y=r^Z$ holds for some~$Z\neq Y$. There must then be an~$N\in\mathbb N$ (which depends only on~$Y$) such that~$\chi_Y(n)=\chi_Y(N)$ holds for all~$n\geq N$. It is straightforward to see that we must have~$Y[N]\neq Z[N]$. We thus get
\begin{equation*}
\mathsf{WKL}_0\vdash\forall f\big(\varphi'(f)\to\exists^iY:r^Y\notin D_f\big).
\end{equation*}
In order to apply Theorem~\ref{thm:isolated-existence}, we must check that~$\varphi'$ is~$\Sigma^1_1$. The crucial condition that $f$ is unbounded on~$[0,1]\cap D_f$ is equivalent to the arithmetical statement
\begin{equation*}
\forall N\in\mathbb N\,\exists q\in\mathbb Q\,\big(q\in[0,1]\cap D_f\land |f(q)|\geq N\big).
\end{equation*}
For the crucial direction of this equivalence, we assume that we have~$f(x)>N$ for some real~$x\in[0,1]\cap D_f$. Pick a rational~$\varepsilon\leq(f(x)-N)/2$. In terms of the representation $f=f_\Phi$ that was discussed above, we obtain a tuple $(n,a,\delta,b,\varepsilon)\in\Phi$ with $|x-a|<\delta$ and $|f(x)-b|\leq\varepsilon$. Take positive~$q,\eta\in\mathbb Q$ with $B_\eta(q)\subseteq B_\delta(a)\cap[0,1]$. As the singularities of~$f$ are isolated, we may assume~$q\in D_f$. By the coherence conditions (see Definition~II.6.1 of~\cite{simpson09}), there is an $n'$ with~$(n',q,\eta,b,\varepsilon)\in\Phi$. We get
\begin{equation*}
|f(q)-f(x)|\leq|f(q)-b|+|b-f(x)|\leq 2\varepsilon\leq f(x)-N
\end{equation*}
and hence~$f(q)\geq f(x)-|f(q)-f(x)|\geq N$. Let us also observe that~$r^Y\notin D_f$ is a $\Sigma^0_2$-statement. By Theorem~\ref{thm:isolated-existence} (even by the special case for which we have given a new proof), we thus obtain
\begin{equation*}
\mathsf{RCA}_0\vdash\forall f\big(\varphi'(f)\to\exists^iY:r^Y\notin D_f\big).
\end{equation*}
By contraposition, $\mathsf{RCA}_0$ proves that any continuous function~$f:D_f\to\mathbb R$ with~$\varphi(f)$ and~$[0,1]\subseteq D_f$ is bounded on~$[0,1]$. Apart from the fact that we have focused on the interval~$[0,1]$ for convenience, this is the conclusion of the theorem.
\end{proof}

\bibliographystyle{amsplain}
\bibliography{WKL-conservative}

\end{document}